\documentclass[10pt,reqno]{amsart}
  \usepackage{geometry}
\usepackage{amssymb,enumerate,color}
\usepackage{color, mathrsfs}
\usepackage{latexsym,amsbsy,bm}

\usepackage{comment}

\usepackage{enumitem}

\input epsf

\usepackage{placeins}
\usepackage{footnote}

\newcommand {\SC} {{\mathbb C}}

\newcommand {\SN} {{\mathbb N}}
\newcommand {\SR} {{\mathbb R}}
\newcommand {\SRd} {{\SR^d}}

\let\oldphi=\phi

\renewcommand {\phi} {{\varphi}}

\newcommand {\al} {{\alpha}}
\newcommand {\dt} {{\delta}}
\newcommand {\Dt} {{\Delta}}
\newcommand {\e} {{\varepsilon}}
\newcommand {\ga} {{\gamma}}

\newcommand {\la} {{\lambda}}

\newcommand{\cL}{\mathcal L}
\newcommand{\cO}{\mathcal O}

\newcommand{\cK}{\mathcal K}
\newcommand{\cR}{\mathcal R}

\def\Dom{\mathop{\rm Dom}}

\numberwithin{equation}{section}
\newtheorem{theorem}{Theorem}[section]
\newtheorem{lemma}[theorem]{Lemma}

\newtheorem{corollary}[theorem]{Corollary}
\newtheorem{Remark}[theorem]{Remark}

\newtheorem{example}[theorem]{Example}
\newtheorem{Question}[theorem]{Question}

\newcommand {\Proofof}[1] {\noindent{\bf P{\footnotesize\bf ROOF} of {#1}: } \ }

\newcommand {\ProofEnd} {
             \begin{flushright} \vskip -0.2in $\Box$ \end{flushright}}

\newcommand{\Ba}[1]{\begin{array}{#1}}
\newcommand{\Ea}{\end{array}}
\newcommand{\Be}{\begin{equation}}
\newcommand{\Ee}{\end{equation}}
\newcommand{\Bea}{\begin{eqnarray}}
\newcommand{\Eea}{\end{eqnarray}}
\newcommand{\Beas}{\begin{eqnarray*}}
\newcommand{\Eeas}{\end{eqnarray*}}
\newcommand{\Benu}{\begin{enumerate}}
\newcommand{\Eenu}{\end{enumerate}}
\newcommand{\Bi}{\begin{itemize}}
\newcommand{\Ei}{\end{itemize}}

\newcommand{\BR}{\begin{Remark} \em}
\newcommand{\ER}{\end{Remark}}
\newcommand{\BE}{\begin{example} \em}
\newcommand{\EE}{\end{example}}
\newcommand{\BQ}{\begin{Question} \em}
\newcommand{\EQ}{\end{Question}}

\newcommand {\Ds} {\displaystyle}

\newcommand {\mand} {{\quad\mbox{and}\quad}}
\renewcommand {\mid} {{\,\,\,\colon\,\,\,}}

\newcommand{\sline}{{\smallskip

\noindent}}

\def\Xint#1{\mathchoice
{\XXint\displaystyle\textstyle{#1}}%
{\XXint\textstyle\scriptstyle{#1}}%
{\XXint\scriptstyle\scriptscriptstyle{#1}}%
{\XXint\scriptscriptstyle\scriptscriptstyle{#1}}%
\!\int}
\def\XXint#1#2#3{{\setbox0=\hbox{$#1{#2#3}{\int}$ }
\vcenter{\hbox{$#2#3$ }}\kern-.6\wd0}}

\def\mint{\Xint-}

\usepackage{bbm}
\def\bbone{{\mathbbm 1}}

\newcommand{\cM}{\mathcal M}

\newcommand {\bB} {{\overline{B}}}

\newcommand {\vrho} {{\varrho}}

\newcommand {\SP} {{\mathbb P}}

\newcommand{\fS}{\mathfrak S}

\begin{document}

\title[Lebesgue points and non-tangential convergence]{Lebesgue points of measures and non-tangential convergence of Poisson-Hermite integrals}

\author[Flores, Garrig\'os,  Viviani]{G. Flores, G. Garrig\'os,    B. Viviani}

\address{G. Flores, CIEM-CONICET \& FaMAF, Universidad Nacional de C\'ordoba, Av. Medina Allende s/n, 
Ciudad Universitaria, CP:X5000HUA C\'ordoba, Argentina}
\email{guilleflores@unc.edu.ar}

\address{G. Garrig\'os, Departamento de Matem\'aticas, Universidad de Murcia, 30100, Espinardo, Murcia, Spain}
\email{gustavo.garrigos@um.es}

\address{Beatriz Viviani\\
IMAL (UNL-CONICET) y FIQ (Universidad Nacional del Litoral), Colectora Ruta Nac. N 168, Paraje El Pozo - 3000 Santa Fe - Argentina} \email{viviani@santafe-conicet.gov.ar}

\thanks{G.G. was supported in part by grants PID2022-142202NB-I00 
from Agencia Estatal de Investigaci\'on (Spain), and grant 21955-PI22 from Fundaci\'on S\'eneca (Regi\'on de Murcia, Spain).
 }
 
 \thanks{F.G. was partially supported by grants from CONICET and SECYT-UNC.
 }

\date{\today}


\begin{abstract}

We study differentiability conditions on a complex measure $\nu$ at a point $x_0\in\SR^d$, 
in relation with the boundary convergence at that point of the Poisson-type integral $\SP_t\nu=e^{-t\sqrt L}\nu$, where $L=-\Dt+|x|^2$ is the Hermite operator. 
In particular, we show that $x_0$ is a Lebesgue point for $\nu$ iff a slightly stronger notion than non-tangential convergence holds for $\SP_t\nu$ at $x_0$.
We also show non-tangential convergence when $x_0$ is a $\sigma$-point of $\nu$, a weaker notion than Lebesgue point, which for $d=1$
coincides with the classical Fatou condition.
\end{abstract}

\maketitle

\section{Introduction}


Let $\nu$ be a locally finite complex measure in $\SR^d$, meaning that $\nu=\nu_1+i\nu_2$ and $\nu_1$, $\nu_2$ are signed 
Borel measures which are finite in compact sets.

We shall use the notion of Lebesgue point for $\nu$ as defined by Saeki in \cite{Sae96}.
Namely, we say that $x_0\in\SR^d$ is a Lebesgue point of $\nu$, denoted $x_0\in\cL_\nu$, 
if there exists $\ell\in\SC$ such that
\Be
\label{v1}
\lim_{r\to0}\frac{|\nu-\ell\,dy|\big(B_r(x_0)\big)}{|B_r(x_0)|}=0.
\Ee
Here $dy(E)=|E|$ denotes the standard Lebesgue measure of a  set $E\subset\SR^d$ and $B_r(x_0)=\{x\in\SR^d\mid |x-x_0|<r\}$.
If we write $\nu$ in terms of the Lebesgue-Radon-Nikodym decomposition, that is
\Be
\label{RN}
\nu=f\,dy+\la,\quad \mbox{with $\la\perp dy$}
\Ee
(see e.g. \cite[Thm 3.12]{Fol}), then using the property that $|\mu+\la|=|\mu|+|\la|$ when $\mu\perp\la$, 
we see that $x_0\in \cL_\nu$ if and only if for some $\ell\in\SC$ it holds 
\Be
\lim_{r\to0}\mint_{B_r(x_0)}|f(y)-\ell|\,dy=0\mand \lim_{r\to0}\frac{|\la|\big(B_r(x_0)\big)}{|B_r(x_0)|}=0. 
\label{v2}
\Ee 
In particular, if $\nu=fdx$, this is the usual notion of Lebesgue point for the function $f$, while
for a general $\nu$ as in \eqref{RN} we have
\[
\cL_\nu=\cL_f\cap \cL_\la.
\]
Observe also that, if $x_0\in\cL_\nu$,  we can express the value of $\ell$ in \eqref{v1} without appealing to the 
Lebesgue-Radon-Nikodym decomposition, 
by letting
\Be
\ell=D\nu(x_0):=\lim_{r\to0}\frac{\nu\big(B_r(x_0)\big)}{|B_r(x_0)|}.
\label{Dv}
\Ee
Indeed, this identity follows from the elementary bound
\Beas
\left|\frac{\nu\big(B_r(x_0)\big)}{|B_r(x_0)|}-\ell\right| & = & \left|\frac{(\nu-\ell\,dy)\big(B_r(x_0)\big)}{|B_r(x_0)|}\right|
\leq \frac{|\nu-\ell\,dy|\big(B_r(x_0)\big)}{|B_r(x_0)|}\longrightarrow 0.
\Eeas
The quantity $D\nu(x_0)$ in \eqref{Dv} is sometimes called the \emph{symmetric derivative} of the measure $\nu$ at $x_0$;
see e.g. \cite[Def 7.2]{Rud}.  

\

Consider now the Hermite operator  
\[
L=-\Dt+|x|^2,\quad \mbox{in $\SR^d$},
\]
and its associated Poisson semigroup 
\[
\SP_t=e^{-t\sqrt{L}}=\tfrac{t}{\sqrt{4\pi}}\,\int_0^\infty e^{-\frac{t^2}{4\tau}}\,e^{-\tau L}\,
\,\frac{d\tau}{\tau^{3/2}};
\] see e.g. \cite[Ch 2.2]{St70}. This semigroup (and its close relative involving the Ornstein-Uhlenbeck operator $-\Dt+2x\cdot\nabla$)
have been widely studied in Harmonic Analysis; see e.g. \cite{Muck1, Than93, ST03, LiuSjo16, Urb}.
We wish to find mild differentiability conditions on a measure $\nu$ at an individual point $x_0\in\SR^d$ 
so that the Poisson integrals
\Be
\label{Ptnu}
\SP_t\nu(x)=\int_{\SR^d}\SP_t(x,y)\,d\nu(y), \quad (t,x)\in(0,\infty)\times\SR^d,
\Ee
 have a non-tangential limit when $(t,x)\to(0,x_0)$. 

When $L$ is the Hermite operator, the kernel $\SP_t(x,y)$ can be obtained explicitly, see \eqref{Ptxy} below for a precise expression,
and its growth (for fixed $t$ and $x$) is well determined by the function
\Be
\Phi(y):=\frac{e^{-|y|^2/2}}{(1+|y|)^\frac d2\,[\log(e+|y|)]^\frac32}, \quad y\in\SR^d.
\label{Phi}
\Ee
Namely, it was shown in \cite[Lemma 4.1]{GHSTV} that
for each $t>0$ and $x\in\SR^d$ there exist $c_j(t,x)>0$, $j=1,2$, such that
\Be
\label{c12}
c_1(t,x)\,\Phi(y)\leq \SP_t(x,y)\leq\,c_2(t,x)\,\Phi(y),\quad y\in\SR^d.
\Ee
In particular, let $\cM(\Phi)$ denote the set of (locally finite) complex measures $\nu$ in $\SR^d$
such that 
\[
\int_{\SR^d}\Phi(y)\,d|\nu|(y)<\infty,
\] 
then it follows from \eqref{c12} that, if $\nu\in\cM(\Phi)$, the function $\SP_t\nu(x)$ in \eqref{Ptnu} is well-defined for all $t>0$ and $x\in\SR^d$.
Moreover,  $u(t,x)=\SP_t\nu(x)$ is smooth in $\SR^{d+1}_+=(0,\infty)\times\SR^d$ 
and satisfies the PDE
\[
u_{tt}=-\Dt_x u +|x|^2u,
\]
see e.g. \cite[Theorem 1.3]{FV23}. 

\

Our first result in this paper investigates the relation between the Lebesgue point condition, 
$x_0\in\cL_\nu$, and the existence of the non-tangential limit
\[
\lim_{
|x-x_0|<\al t\to 0} \SP_t\nu(x),
\]
for every $\al>0$. We shall prove the following

\begin{theorem}\label{th1}
Let $\nu\in\cM(\Phi)$ and $x_0\in\SR^d$. Then, the following assertions are equivalent
\Benu
\itemsep0.2cm
\item[(i)] \quad $x_0\in\cL_\nu$;

\item[(ii)] \quad $\Ds \lim_{t\to 0} \SP_t\big(|\nu-\ell\,dy|\big)(x_0)=0$, for some $\ell\in\SC$;

\item[(iii)] \;$\Ds \lim_{|x-x_0|<\al t\to 0} \SP_t\big(|\nu-\ell\,dy|\big)(x)=0$, for some $\ell\in\SC$ and some (all) $\al>0$.
\Eenu
Morever, if these assertions hold we can take $\ell=D\nu(x_0)$, and for every $\al>0$ it also holds
\Be
\lim_{|x-x_0|<\al t\to 0} \SP_t\nu(x) = D\nu(x_0).
\label{nt}
\Ee
\end{theorem}

As a second result, we give a weaker condition than Lebesgue point which still ensures non-tangential convergence. 
This is the notion of $\sigma$-point introduced by V. Shapiro in \cite{Sha06}; see also \cite{Sarkar21}. 
We say that $x_0\in\SR^d$ is a $\sigma$-point of a (locally finite) complex measure $\nu$, denoted $x_0\in\fS_\nu$, if
there exists $\ell\in\SC$ such that
\Be
\lim_{|x-x_0|+r\to0} \frac{\big|(\nu-\ell\,dy)(B_r(x))\big|}{(|x-x_0|+r)^d}=0.
\label{limsig}
\Ee
Note that, if $x_0$ is a $\sigma$-point, then $\nu$ has a symmetric derivative at $x_0$, and
we can take $\ell=D\nu(x_0)$ in \eqref{limsig}. 
This follows by just restricting the above limit to $x=x_0$. 
Also, since $B_r(x)\subset B_{r+|x-x_0|}(x_0)$, we have
\[
\cL_\nu\subset \fS_\nu,
\]
and the inclusion can be strict in view of the examples in \cite[\S3]{Sha06}.
Finally, when $d=1$ there is a simple characterization: write $\nu$ as a Lebesgue-Stieltjes measure $\nu=dm_F$, with
\[
\nu\big((a,b]\big)=F(b)-F(a);
\] 
see e.g. \cite[Thm 1.16]{Fol}. Then $x_0\in \fS_\nu$ iff $F$ is differentiable at $x_0$, in which case
\[
F'(x_0)=\ell=D\nu(x_0).
\]
The proof is elementary; see also \cite[Proposition 3.4]{Sarkar21}.

Our second result has the following statement.

\begin{theorem}\label{th2}
Let $\nu\in\cM(\Phi)$ and $x_0\in\SR^d$ a $\sigma$-point of $\nu$. Let $\ell=D\nu(x_0)$. 

\sline (i) If $d\in\{1,2,3\}$ then
\Be
\lim_{|x-x_0|<\al t\to 0} \SP_t\nu(x) = \ell 
, \quad \forall\,\al>0.
\label{nt2}
\Ee
\sline (ii) If $d\geq4$ then \eqref{nt2} holds if it is additionally assumed that
\Be
\lim_{r\to0}\frac{|\nu-\ell\, dy|(B_r(x_0))}{r^{d-3}}=0.
\label{d3}
\Ee
\end{theorem}

As a corollary, when $d=1$ we obtain the non-tangential convergence under the classical Fatou condition, see \cite{Fat06} or 
\cite[Thm III.7.9.ii]{Zyg}. This result seems to be new for Poisson-Hermite integrals.

\begin{corollary}\label{c1}
Let $d=1$ and $\nu=m_F$ be a Lebesgue-Stieljes measure in $\cM(\Phi)$. If $F$ is differentiable in $x_0$ then
\[
\lim_{|x-x_0|<\al t\to 0} \SP_t\nu(x) = F'(x_0), \quad \forall\,\al>0.
\]
\end{corollary}

\

Finally, we remark that, when $x_0=0$, one can slightly relax the condition \eqref{d3} (and even remove it, if $d=4$), due to
the special symmetry of the resulting kernel $\SP_t(0,y)$; see a precise statement in Theorem \ref{th3} below.

\

We comment on previous results on these matters for the Poisson-Hermite semigroup. When the initial datum is a function, that is $\nu=f\,dy$, 
generic results on a.e. convergence of $\SP_tf$ go back to \cite{Muck1}, see also \cite{ST03, GHSTV}. These references contain the main kernel estimates, but do not consider
convergence at individual Lebesgue points of $f$. A statement for these was recently given in \cite{GHSV} (for the related Laguerre operator),
and requires an argument that we elaborate further here in the new Lemma \ref{L_la}. Concerning measures $\nu$ as initial data, we are only aware of 
the a.e. results for $\SP_t\nu$ in \cite[Theorem 1.3]{FV23}, which again do not consider the behavior at individual points.

\

Regarding Theorem \ref{th2}, no results involving weaker notions than Lebesgue points seem to appear in the literature
for Poisson-Hermite integrals, even when $\nu$ equals $f$ or $d=1$. 
However, there is an extensive literature in the classical setting of harmonic functions
regarding optimal conditions for non-tangential convergence; see e.g. \cite{Loo43,RaUl88, BrCh90}, or the more recent
\cite{Sae96, Sha06, Sarkar21} and references therein.
The latter papers consider a general setting of approximations of the identity,
but always associated with convolution kernels with a dilation structure, that is, 
\[
K_t(x,y)=t^{-d}\oldphi((x-y)/t), \quad t>0,
\]
and most often with a radially decreasing $\oldphi$.

One novelty here regards the fact that one must consider kernels $\SP_t(x,y)$ without this convolution structure, 
so different arguments are necessary to carry out the proofs. Our arguments will be based on 
a suitable kernel decomposition into radial and non-radial parts, and very precise estimates on the kernel
which are optimal for dimensions $d\leq3$, but require the additional condition \eqref{d3} when $d\geq4$.
We do not know whether Theorem \ref{th2} may hold without this hypothesis (if $x_0\not=0$), although
it seems unlikely in view of examples in \cite{FGV2}, where a similar condition in relation with the \emph{normal} 
convergence of $\SP_t\nu(x_0)$ appears for dimensions $d\geq4$, and cannot be removed in that setting.
We nevertheless remark that \eqref{d3} is quite mild (certainly weaker than the Lebesgue point condition),
due to the factor $r^{d-3}$ in the denominator.

\

The paper is structured as follows. In section \ref{S_2} we compile the notation and the required kernel estimates,
and present the proof of Theorem \ref{th1}. We establish in \S\ref{S_2-3} a general Lemma \ref{L_la}, which may have 
an independent interest. Finally, in section \ref{S_3} we give a detailed proof of Theorem \ref{th2}.
\

\section{Proof of Theorem \ref{th1}}\label{S_2}

It is clear that (iii) implies (ii). We shall show in separate subsections the other implications.
Below we use the notation
\[
p_t(z)=
\frac{t}{(t^2+|z|^2)^{\frac{d+1}2}},
\]
which except for a multiplicative constant is the standard Poisson kernel.
Recall also that the Poisson kernel $\SP_t(x,y)$ associated with the Hermite operator
can be explicitly given by
\Be
\label{Ptxy}
\SP_t(x,y)=\,c_d\,t\,\int_0^1\frac{e^{-\frac{t^2}{\ell(s)}}\,(1-s^2)^{\frac d2-1}\,e^{-\frac14(\frac{|x-y|^2}s+s|x+y|^2)}}
{s^{\frac d2}\,\big(\ell(s)\big)^{3/2}}\,ds,
\Ee
where $\ell(s)=2\ln\frac{1+s}{1-s}$ and $c_d
>0$; see e.g. \cite[(4.1)]{GHSTV}. Note that $\ell(s)\approx s$ when $s\in(0,1/2)$, a fact that we shall use often below.

\subsection{Proof of $(ii)\Longrightarrow(i)$}

We have the following estimate from below for the kernel $\SP_t(x,y)$.

\begin{lemma}\label{L0}
Given $R\geq1$, there exists a constant $\dt_R>0$ such that
\Be
\label{pbelow}
\SP_t(x,y)\,\geq\, \dt_R\;p_t(x-y), \quad \mbox{when $|x|,|y|, t\leq R$}.
\Ee
\end{lemma}
\begin{proof}
If we restrict the range of integration in \eqref{Ptxy} to $s\in(0,1/2)$, so that 
$\ell(s)\approx s$, after disregarding inessential terms we have
\Beas
\SP_t(x,y)  & \gtrsim & t\,\int_0^{1/2} \frac{ e^{-\frac{c\,t^2}{4s}} e^{-\frac14(\frac{|x-y|^2}s+s|x+y|^2)} }{ s^{\frac{d+3}2} } 
\, ds\\
& \geq & e^{-\frac{R^2}2}\, t\,\int_0^{1/4} \frac{ e^{-\frac{c\,t^2+|x-y|^2}{4s}} 
 }{ s^{\frac{d+1}2} } 
\, \frac{ds}{s}\\
& = &  \frac {c_R\,t}{(ct^2+|x-y|^2)^{\frac{d+1}2}}\,
\int_{ct^2+|x-y|^2}^\infty e^{-u} u^{\frac{d+1}2}\,\frac{du}u,
\Eeas
performing in the last step the change of variables $u=(c\,t^2+|x-y|^2)/(4s)$.
Using again the assumption $|x|, |y|, t\leq R$, the last integral is bounded below by
\[
\int_{(4+c)R^2}^\infty e^{-u} u^{\frac{d+1}2}\,\frac{du}u=c'_R.
\]
Combining these expressions one obtains \eqref{pbelow}.
\end{proof}

Going back to the proof of Theorem \ref{th1}, assume that (ii) holds, and denote $|\nu-\ell dx|=\mu$.
Pick $R\geq|x_0|+1$ and $t\leq 1$. Then
the previous lemma gives
\Beas
\SP_t\big(\mu\big)(x_0) & \gtrsim &  \dt_R\,\int_{|y-x_0|\leq t} p_t(x_0-y)\,d\mu(y) .
\Eeas
Since $p_t(x_0-y)\approx t^{-d}$ when $|x_0-y|\leq t$, this implies
\[
\SP_t\big(\mu\big)(x_0) \, \gtrsim\,\frac{\mu\big(B_t(x_0)\big)}{|B_t(x_0)|}.
\]
By (ii), the left-hand side goes to 0 as $t\to0$, hence so does the right hand side. 
Since $\mu=|\nu-\ell\,dx|$, this implies that $x_0$ is a Lebesgue point of $\nu$ (and therefore, also that $\ell=D\nu(x_0)$).

\subsection{Proof of $(i)\Longrightarrow(iii)$}

We shall use the following upper bound for $\SP_t(x,y)$, which can be found in \cite[Lemma 4.2]{GHSTV}.

\begin{lemma}\label{L1}
There exists a constant $\ga\geq2$ and a continuous function $C(x)$ such that
\[
\SP_t(x,y)\leq C(x)\,\Big[p_t(x-y)\bbone_{|y|\leq \gamma\max\{|x|,1\}}\,+\,t\,\Phi(y)\Big],
\]
for all $t>0$ and $x,y\in\SR^d$.
\end{lemma}

We also quote the following elementary lemma.

\begin{lemma}\label{L2}
Let $\al>0$. Then, there exists $C_\al>0$ such that
\Be
\label{cal}
\frac1{C_\al}\,\big(|x_0-y|+t\big)\leq |x-y|+t\leq C_\al\,\big(|x_0-y|+t\big),
\Ee
whenever $|x-x_0|\leq\al t$ and $y\in\SR^d$.
\end{lemma}
\begin{proof}
The proof is elementary using the condition $|x-x_0|\leq\al t$ and the triangle inequality. 
Indeed, on the one hand
\[
|x-y|+t\leq |x-x_0|+|x_0-y|+t\leq |x_0-y|+(\al+1)t,
\]
which implies the upper bound in \eqref{cal} with $C_\al=\al+1$.
The lower bound follows from the upper one after interchanging the roles of $x$ and $x_0$.  
\end{proof}

We turn to the proof of (iii) in Theorem \ref{th1}. Since we are interested in non-tangential limits at $x_0$, we fix $\al>0$ and
consider only points $(t,x)$ such that $|x-x_0|\leq\al t\leq 1$. For such points, Lemmas \ref{L1} and \ref{L2} imply that
\[
\SP_t(x,y)\lesssim C_{x_0}\,\Big[p_t(x_0-y)\bbone_{|y|\leq \gamma(|x_0|+2)}\,+\,t\,\Phi(y)\Big],\quad \forall\,y\in\SR^d,
\]
for some constant $C_{x_0}>0$ (which we could take $\max_{|x|\leq|x_0|+1}C(x)$).
For simplicity, we write
\[
K=\big\{y\in\SR^d\mid |y|\leq\ga(|x_0|+2)\big\},
\]
and as before we denote  $\mu=|\nu-\ell dx|$. The previous inequalities then imply
\Be
\SP_t\mu(x)\,\lesssim\,C_{x_0}\,\Big[\int_{K}p_t(x_0-y)\,d\mu(y)\,+\,t\,\int_{\SR^d}\Phi(y)\,d\mu(y)\Big].
\label{al1}
\Ee
The assumption that $\nu\in\cM(\Phi)$ implies that the last summand is $O(t)$ as $t\to 0$, so we must only
estimate the first term
\[
\int_{K}p_t(x_0-y)\,d\mu(y).
\]
Here one could use the known results about the standard Poisson kernel to conclude that
\Be
\label{lim1}
\lim_{t\to0}\int_{K}p_t(x_0-y)\,d\mu(y)=0,
\Ee
under the assumption $x_0\in\cL_\nu$. We briefly sketch the argument for completeness. 
Given $\e>0$, we choose $\dt>0$ such that
\Be
\label{eps}
\frac{\mu\big(B_r(x_0)\big)}{r^d}<\e, \quad r\in(0,2\dt].
\Ee
When $|y-x_0|\geq\dt$ we have $p_t(x_0-y)\leq t/\dt^{d+1}$, and hence
\Be
\label{al2}
\int_{y\in K\setminus B_\dt(x_0)}p_t(x_0-y)\,d\mu(y)\leq \,\frac{t\,\mu(K)}{\dt^{d+1}}<\e,
\Ee
if $t<t_0:=\min\{\e\,\dt^{d+1}/\mu(K),\dt\}$. On the other hand, given any such $t$, we can find an integer $J\in\SN$ such that
\[
\dt<2^Jt\leq 2\dt.
\]
So letting $S_j:=B_{2^{j}t}(x_0)\setminus B_{2^{j-1}t}(x_0)$ and using that $p_t(x_0-y)\lesssim t^{-d}2^{-j(d+1)}$ in $S_j$ 
we have
\Beas
\int_{B_\dt(x_0)}p_t(x_0-y)\,d\mu(y) & \lesssim & \int_{B_t(x_0)}t^{-d}\,d\mu(y) +\sum_{j=1}^J
\int_{S_j}(2^jt)^{-d}2^{-j}\,d\mu(y)\\
& \leq & \frac{\mu(B_t(x_0))}{t^d}+\sum_{j=1}^J 2^{-j}\frac{\mu(B_{2^jt}(x_0))}{(2^jt)^d}\\
& \leq & 2\e,
\Eeas
using \eqref{eps} in the last step. This shows \eqref{lim1}.

\subsection{Last assertion in Theorem \ref{th1}}\label{S_2-3}
It remains to prove that, under the assumption (iii), it also holds
\[
\lim_{|x-x_0|<\al t\to 0} \SP_t\nu(x) = \ell.
\] 
This will be a consequence of the following lemma.

\begin{lemma}
\label{L_la}
Let $\nu\in\cM(\Phi)$ and $\ell\in\SC$. Then the function
\[
E(t,x):=\,\big(\SP_t\nu(x)-\ell\big) - \SP_t(\nu-\ell\,dy)(x) 
\]
satisfies
\[
\lim_{(t,x)\to(0,x_0)} E(t,x)=0, \quad \forall\,x_0\in\SR^d.
\]
\end{lemma}

Indeed, assuming the lemma and using (iii) we have
\[
\big|\SP_t\nu(x)-\ell\big|\leq \SP_t|\nu-\ell\,dy|(x) + |E(t,x)| \longrightarrow0,
\]
when $|x-x_0|< \al t\to 0$.

\

We now prove Lemma \ref{L_la}. We remark that this is not completely straightforward since for the Hermite operator we do not have 
$\SP_t(\bbone)=\bbone$ (since $\la=0$ is not an eigenvalue of $L$).
We shall use instead an argument, borrowed from \cite[\S3.2]{FGSV} (see also \cite[\S6]{GHSV}), which can be adapted to more general 
positive self-adjoint operators $L$. 
The only tool is the existence of a \emph{regular positive eigenvector} of $L$, that is $\psi\in\Dom(L)$ such that
\Benu
\item[(a)] $\psi\in C^\infty(\SRd)$
\item[(b)] $\psi(x)>0$, $\forall\,x\in\SR^d$
\item[(c)] $L(\psi)=\la\psi$, for some $\la\geq0$.
\Eenu
When $L=-\Dt+|x|^2$, it is elementary to check that
\Be
\label{psi}
\psi(x)=e^{-|x|^2/2}
\Ee
satisfies these properties with $\la=d$.

\Proofof{Lemma \ref{L_la}}
Let $\psi(x)$ be as in \eqref{psi}. Since $\SP_t=e^{-t\sqrt{L}}$, we also have
\[
\SP_t\psi=e^{-t\sqrt\la}\psi.
\]
Then
\Be\label{Pt_aux}
\SP_t\nu(x)-\ell\,=\, \big(\SP_t\nu(x)-\ell\,e^{-t\sqrt{\la}}\big)\,+\,\big(e^{-t\sqrt{\la}}-1\big)\,\ell. 
\Ee
The last summand goes to 0 as $t\to0$, so we look at the first summand.
Since $\psi(x_0)>0$, we may as well consider
\[
A(t,x):=\psi(x_0)\,\big(\SP_t\nu(x)-\ell\,e^{-t\sqrt{\la}}\big), 
\]
which using that $\SP_t\psi=e^{-t\sqrt{\la}}\psi$ we can write as
\Beas
A(t,x) 
& = & \big[\psi(x_0)\,\SP_t\nu(x)-\ell\,\SP_t\psi(x)\big]\,+\,\ell\,\big[\SP_t\psi(x)-\SP_t\psi(x_0)\big]\\
& = & A_1(t,x)+A_2(t,x).
\Eeas
Now,
\[
|A_2(t,x)|=|\ell|\,e^{-t\sqrt{\la}}|\psi(x)-\psi(x_0)| \longrightarrow0 ,\quad \mbox{if $|x-x_0|\to0$},
\]
 by the continuity of $\psi$. Finally, 
\Beas
A_1(t,x) & = & \psi(x_0)\,\SP_t(\nu-\ell\,dy)(x)\,-\,\ell\,\SP_t\big(\psi-\psi(x_0)\big)(x)\\
  & = & A_{1,1}(t,x)+A_{1,2}(t,x).
\Eeas
Since the function $g=\psi-\psi(x_0)$ is continuous everywhere (and vanishes at $x_0$), using Lemma \ref{L1} and standard results 
on approximations of the identity one can show that 
\[
|A_{1,2}(t,x)| \leq |\ell| \, \SP_t|g|(x)\,\lesssim\,\big(p_t*|g|\big)(x)+t\int|g|\Phi\longrightarrow0, 
\]
as $|x-x_0|+t\to 0$. Thus,
\[
A(t,x)=A_{1,1}(t,x) + O(1).
\]
Dividing this expression by $\psi(x_0)$, and going back to \eqref{Pt_aux}, 
one obtains
\[
\SP_t\nu(x)-\ell\,= \SP_t(\nu-\ell\,dy)(x) + O(1).
\]
This proves Lemma \ref{L_la}, and hence it concludes the proof of Theorem \ref{th1}.

\ProofEnd

\section{Proof of Theorem \ref{th2}}\label{S_3}

It will suffice to consider the case when $\ell=0$; otherwise, one would apply the reasoning to the measure $\mu=\nu-\ell\,dy$ 
together with Lemma \ref{L_la}. 
We fix $\al>0$ and must show that
\Be
\label{limP}
\lim_{|x-x_0|<\al t\to 0} \SP_t\nu(x) = 0.
\Ee
By the condition $x_0\in\fS_\nu$ (with $\ell=0$), for any fixed $\e>0$ there exists $\dt\in(0,1)$ such that
\Be
\big|\nu\big(B_r(x)\big)\big|\leq \e\,(|x-x_0|+r)^d, \quad \mbox{whenever}\;\;|x-x_0|+r< 2\dt.
\label{E}
\Ee
If the dimension $d\geq4$, we also assume that $\dt$ is chosen such that
\Be
|\nu|\big(B_r(x_0)\big)\,<\,\e\; r^{d-3}, \quad \mbox{when}\;\;r<2\dt,
\label{E4}
\Ee
in view of the hypothesis in \eqref{d3}. 

The first part of the proof is similar to the proof of Theorem \ref{th1}: 
we consider only points $(t,x)$ such that $|x-x_0|\leq\al t\leq 1$; in particular $|x|\leq |x_0|+1$. 
By Lemma \ref{L1} we have
\[
\SP_t(x,y)\lesssim C_{x_0}\,\Big[p_t(x-y)\bbone_{|y|\leq \gamma(|x_0|+2)}\,+\,t\,\Phi(y)\Big],\quad \forall\,y\in\SR^d,
\]
for some constant $C_{x_0}>0$. 
In the region $|x-y|\geq\dt$, 
we have $p_t(x-y)\lesssim t/\dt^{d+1}$, so the above bound gives
\Bea
\int_{|x-y|\geq\dt}\SP_t(x,y)\,d\,|\nu|(y) & \lesssim & C'_{x_0}\,\big(\dt^{-d-1}+\int\Phi\,d|\nu|\big)\cdot t\nonumber\\
&  = &  c(x_0,\dt)\cdot t\,<\,\e,\label{e0}
\Eea
provided we assume $t<t_0:=\min\{\e/c(x_0,\dt), \dt\}$. So, in the remainder of the proof we shall aim to show 
\[
\Big|\int_{|x-y|< \dt}\SP_t(x,y)\,d\,\nu(y)\Big|=\cO(\e),
\]
when $|x-x_0|\leq \al t$ and $t$ is sufficiently small. Since $|x|\leq |x_0|+1$ we have 
\[
B_\dt(x)\subset\{y\in\SR^d\mid |y|\leq |x_0|+2\}=:K,
\]
so in the sequel we may assume $\nu$ to be supported in the compact set $K$.

We next remove another inessential term: we define
the ``local'' part of the Poisson kernel by
\Be
\label{Pt0}
\SP_t^{0}(x,y):=\,c_d\,t\,\int_0^{1/2}\frac{e^{-\frac{t^2}{\ell(s)}}\,(1-s^2)^{\frac d2-1}\,e^{-\frac14(\frac{|x-y|^2}s+s|x+y|^2)}}
{s^{\frac d2}\,\big(\ell(s)\big)^{3/2}}\,ds.
\Ee
Then we have
\[
\SP_t^{1}(x,y):=\SP_t(x,y)-\SP_t^{0}(x,y)\,\leq\, C\, t\,\int_{1/2}^1(1-s^2)^{\frac d2-1}ds=C'\,t,
\]
for some $C'>0$ (depending only on $d$), and therefore
\Be
\label{P1s}
\int_{K}\SP_t^{1}(x,y)\,d|\nu|(y)\,\leq\,C'\,|\nu|(K)\cdot t= C''\, t<\e,
\Ee
if $t<t_1=\min\{t_0, \e/C''\}$. So it suffices to show that
\[
A(t,x):=\int_{|x-y|< \dt}\SP^0_t(x,y)\,d\,\nu(y) =\cO(\e)
\]
when $|x-x_0|\leq \al t$ and $t$ is sufficiently small. We fix $x$ in what follows, and changing variables $y=x+h$, we rewrite the above expression as
\[
A(t,x)=\int_{|h|< \dt}\SP^0_{t,x}(h)\,d\,\nu_x(h),
\]
where for simplicity we denote
\[
\SP^0_{t,x}(h):=\SP_t^0(x, x+h)\mand \nu_x(E):=\nu(x+E), \;\;E\subset \SR^d.
\]
The kernel now takes the form
\[
\SP^0_{t,x}(h)=\,c_d\,t\,\int_0^{1/2}\frac{e^{-\frac{t^2}{\ell(s)}}\,(1-s^2)^{\frac d2-1}}
{s^{\frac d2}\,\big(\ell(s)\big)^{3/2}}\,e^{-\frac{|h|^2}{4s}}\,e^{-\frac{s|2x+h|^2}4}\,ds.
\]
This is not a radial function in $h$, so we shall split it into a radial part and a remainder.
To do so, we write
\Beas
e^{-\frac{s|2x+h|^2}4} & = & e^{-s|x|^2}\,e^{-\frac{s|h|^2}4}\, e^{-s\,x\cdot h}\\
& = & e^{-s|x|^2}\,e^{-\frac{s|h|^2}4}\,\big(1+R(s\,x\cdot h)\big),
\Eeas
where letting $u=s\,x\cdot h$, we have
\Be
|R(u)|=|e^{-u}-1|=\Big|\int_0^u e^{-v}\,dv\Big|\leq |u|\,e^{|u|}\leq c_{x_0}\,s\,|x|\,|h|,
\label{Ru}
\Ee
since $s\in(0,1/2)$, $|h|\leq 1$ and $|x|\leq |x_0|+1$. We now split
\Be
\label{split}
\SP^0_{t,x}(h) = \cK_{t,x}(h) +\cR_{t,x}(h),
\Ee
with the ``radial part'' of $\SP^0_{t,x}$ (in the variable $h$) given by
\Be
\label{Ktx}
\cK_{t,x}(h)=\,c_d\,t\,\int_0^{1/2}\frac{e^{-\frac{t^2}{\ell(s)}}\,(1-s^2)^{\frac d2-1}}
{s^{\frac d2}\,\big(\ell(s)\big)^{3/2}}\,e^{-\frac{|h|^2}4(s+\frac1s)}\,e^{-s|x|^2}\,ds.
\Ee
The kernel $\cR_{t,x}(h)$ has a similar expression with an additional (non-radial) factor $R(s\,x\cdot h)$ inside the integral.
Disregarding inessential terms, we can estimate it by
\Bea
\cR_{t,x}(h) & \lesssim & t\,\int_0^{1/2} \frac{e^{-\frac{c\,t^2}{4s}}\,e^{-\frac{|h|^2}{4s}}\,s\,|x|\,|h|} 
{s^{\frac {d+3}2}}\,ds
\label{aux_R1}\\
& = & \frac{t\,|x|\,|h|}{(ct^2+|h|^2)^{\frac{d-1}2}}\,\int_{\frac{ct^2+|h|^2}2}^\infty e^{-v}\,v^{\frac{d-1}2}\,\tfrac{dv}v,\nonumber
\Eea
where in the last step we have changed variables $v=(ct^2+|h|^2)/(4s)$.
At this point we distinguish two cases, if $d>1$, the above integral is bounded by a finite constant, so we have
\Be
\label{A1n}
\big|\cR_{t, x}(h)\big|\lesssim\, \frac{t\,|x|\,|h|}{(ct^2+|h|^2)^\frac{d-1}2}.
\Ee
If $d=1$, using that $t,|h|\ll 1$ we have 
\Be
\label{A11}
\big|\cR_{t, x}(h)\big|\lesssim\, t\,|x|\,|h|\, \log\tfrac1{ct^2+|h|^2}
.
\Ee
Note that in both cases the involved constants depend only on $x_0$ (and the dimension $d$), but not on $x$ or $\dt$.
Note further that 
\[
|x|\leq |x-x_0|+|x_0|\leq \al t+|x_0|\leq 1+|x_0|,
\]
so below we shall absorb the factor $|x|$ in \eqref{A1n} and \eqref{A11} into the constants;
however, in the special case that $x_0=0$ this factor is $\cO(t)$ and will lead to a slightly
better result; see Remark \ref{R_x0} below.

We will now show that
\Be
\label{A1to0}
A_1(t,x):=\int_{|h|< \dt}\cR_{t,x}(h)\,d\,\nu_x(h) =\cO(1), \quad \mbox{as $|x-x_0|< \al t\searrow0$.}
\Ee
The argument is different depending on the dimension.

\sline {\bf Case $d=1$}. In this case we simply have
\Beas
|A_1(t,x)| & \leq & \int_{|h|<\dt}|\cR_{t,x}(h)|\;d|\nu_x|(h)\\
& \lesssim & t\,\log(1/t)\,\int_{|h|<\dt}|h|\;d|\nu_x|(h)\\
& \leq & t\,\log(1/t)\;|\nu|(K),
\Eeas
which vanishes as $t\searrow0$.

For higher dimensions $d>1$, the bound in \eqref{A1n} gives
\Bea
|A_1(t,x)|  & \leq &   \int_{|h|<\dt}|\cR_{t,x}(h)|\;d|\nu_x|(h)\nonumber \\
&  \lesssim &  t\,\int_{|h|<\dt}\frac{|h|}{(t+|h|)^{d-1}}\;d|\nu_x|(h)\nonumber
\\ &  \leq &  t\,\int_{|h|<\dt}\frac{d|\nu_x|(h)}{(t+|h|)^{d-2}}.
\label{auxA1}
\Eea
We again distinguish cases.

\sline {\bf Case $d=2$}. In this case,  \eqref{auxA1} clearly becomes $\cO(t)$ as $t\searrow0$ (with a bound independent of $x$ since 
$|\nu|(K)<\infty$).

\sline {\bf Case $d=3$}. 
We now have   
\Beas
|A_1(t,x)| & \lesssim & \int_{|h|<\dt}\frac{t}{t+|h|}\;d|\nu_x|(h)\,=\,\int_{B_\dt(x)}\frac{t}{t+|y-x|}\;d|\nu|(y)\\
\mbox{{\tiny (by Lemma \ref{L2})}} & \lesssim & \int_{K}\frac{t}{t+|y-x_0|}\;d|\nu|(y)
\Eeas
and the right hand side vanishes as $t\searrow0$ by the dominated convergence theorem (independently of $x$).

\sline {\bf Case $d\geq4$}. In this case, rather than \eqref{A1to0}, we show that
\[
\limsup_{|x-x_0|\leq \al t\to0^+}|A_1(t,x)|\leq C\, \e.
\]
To do so, we shall need the additional hypothesis in \eqref{d3}, in the form \eqref{E4}.
 We break the integral in \eqref{auxA1} into pieces.
Assuming $t\ll\dt/2$, we can find $J\in\SN$ such that $2^{J}t<\dt\leq 2^{J+1}t$.
We then write
\[
\int_{|h|<\dt}\frac{t\;d|\nu_x|(h)}{(t+|h|)^{d-2}}\,dh\leq \int_{|h|<t}\cdots\;+\;\sum_{j=0}^{J}\int_{2^{j}t\leq|h|<2^{j+1}t}\cdots=:I_0+I_1.
\]
Then, using \eqref{E4} one finds that
\[
I_0\leq \frac1{t^{d-3}}\,\int_{|h|<t}\,d|\nu_x| = \frac{|\nu|\big(B_t(x)\big)}{t^{d-3}}<\e,
\]
and 
\[
I_1 \leq  \sum_{j=0}^{J}\int_{|h|<2^{j+1}t}\frac{t\,\;d|\nu_x|}{(2^jt)^{d-2}}\,
= \sum_{j=0}^{J} \frac{2^{-j}}{(2^jt)^{d-3}}|\nu|\big(B_{2^{j+1}t}(x)\big)
\leq C\,\e,
\]
with $C=2^{d-2}$. Combining these estimates with \eqref{auxA1} we obtain
\Be
\label{limA1}
\limsup_{|x-x_0|\leq \al t\to0}|A_1(t,x)|\,\leq\, c_{x_0}\,\e.
\Ee
This finishes the estimate involving the remainder part $\cR_{t,x}(h)$ of the kernel $\SP^0_{t,x}(h)$; see \eqref{split}.
We remark that the hypothesis \eqref{nt2} has not been used in this part. 

\

We now turn to estimate 
the piece involving the radially decreasing part $\cK_{t,x}(h)$, which makes a crucial use of this hypothesis (in the form \eqref{E}). 
We shall need a lemma from measure theory, whose proof is similar to \cite[Proposition 6.23]{Fol}, but that we sketch for completeness.
This gives a formula of polar coordinates \eqref{intmu} for a general measure $\mu$.

\begin{lemma}\label{Lm}
Given a (locally finite) complex measure $\mu$ in $\SR^d$, let 
\Be
\vrho(r):=\mu\big(\bB_r(0)\big)=\mu\Big(\big\{|x|\leq r\big\}\Big), \quad r\geq0,
\label{vrho}
\Ee
and denote by $m_\vrho$ the associated Lebesgue-Stieltjes measure in $[0,\infty)$, that is
\Be
m_\vrho\big((a,b]\big)=\vrho(b)-\vrho(a), \mand m_\rho\big(\{0\}\big)=\vrho(0).
\label{mvrho}
\Ee
Then, for every  radial $f(x)=f_0(|x|)\in L^1(\SR^d;|\mu|)$, it holds
\Be
\int_{\SR^d}f(x)\,d\mu(x)=\int_{[0,\infty)}f_0(r)\,dm_\vrho(r).
\label{intmu}
\Ee
\end{lemma}
\begin{proof}
We may assume that the measure $\mu$ is positive. Consider the new 
measure $\mu_0$ defined on Borel sets $A\subset[0,\infty)$ by
\[
\mu_0(A):=\mu\Big(\big\{x\mid |x|\in A\big\}\Big)=\int_{\SR^d}\bbone_A(|x|)d\mu(x).
\]
When $A$ is an interval $(a,b]\subset(0,\infty)$ (or $A=\{0\}$), it is clear by definition that
\[
\mu_0(A)=m_\vrho(A).
\]
Since these elementary sets generate the $\sigma$-algebra of all Borel sets in $[0,\infty)$, 
by uniqueness, the two measures $\mu_0$ and $m_\vrho$ will coincide over all such sets; see e.g. \cite[Theorem 1.14]{Fol}.
Thus, for every function of the form $f(x)=\bbone_A(|x|)$ with $A$ a Borel set, it will hold
\[
\int_{\SR^d}f(x)\,d\mu(x)=\int_{[0,\infty)}f_0(r)\,d\mu_0(r)=\int_{[0,\infty)}f_0(r)\,dm_\vrho(r).
\]
By linearity, this identity extends to all simple functions, and by monotone convergence to all non-negative
$f(x)=f_0(|x|)$. A further extension by linearity to complex functions in $L^1(d|\mu|)$ establishes the result.
\end{proof}

We shall apply the previous lemma to the measure $\mu=\nu_x$, so that
\[
A_2(t,x):= \int_{|h|< \dt}\cK_{t,x}(h)\,d\,\nu_x(h)=\int_{[0,\dt)}\cK_{t,x}^0(r)\,dm_{\vrho_x}(r),
\]
with $\cK_{t,x}(h)=\cK_{t,x}^0(|h|)$ and
\Be
\vrho_x(r):=\nu_x\big(\bB_r(0)\big)=\nu\big(\bB_r(x)\big).
\label{vrhox}
\Ee
Notice that condition \eqref{E} implies
\Be\label{Erho}
\big|\vrho_x(r)\big|\leq \e\,(|x-x_0|+r)^d, \quad \mbox{if}\;\;|x-x_0|+r<2\dt.
\Ee
Then, integrating by parts (see e.g. \cite[Theorem 3.36]{Fol}) we have
\Bea
A_2(t,x) & = & \cK^0_{t,x}(0)\vrho_x(0)+\int_{(0,\dt)}\cK^0_{t,x}(r)\,dm_{\vrho_x}(r)\nonumber\\
& = & \cK^0_{t,x}(0){\vrho_x}(0)+ \Big[\cK^0_{t,x}(r){\vrho_x}(r)\Big]_{r=0}^{r=\dt^-}-\int_0^\dt{\vrho_x}(r)\,\frac{d\cK^0_{t,x}}{dr}(r)\,dr\nonumber\\
& = & \cK^0_{t,x}(\dt){\vrho_x}(\dt^-)+\int_0^\dt{\vrho_x}(r)\,\Big|\frac{d\cK^0_{t,x}}{dr}(r)\Big|\,dr,\label{bdry}
\Eea
using in the last line that (for each fixed $t$ and $x$)
the function 
\[
r\longmapsto \cK^0_{t,x}(r)
\]
is decreasing in $[0,\infty)$; see \eqref{Ktx}. Note also that
\Be
\label{Kro}
\cK_{t,x}(h)\lesssim \frac{t}{(t+|h|)^{d+1}},
\Ee
which can be proved with a similar argument as we did in \eqref{aux_R1} for the kernel $\cR_{t,x}(h)$ (removing the factor $s\,|x|\,|h|$
that appears in that kernel due to \eqref{Ru}). So, using this and \eqref{Erho}, the boundary term in \eqref{bdry} satisfies
\[
\cK^0_{t,x}(\dt)\,\big|\vrho_x(\dt^-)\big|=\cK^0_{t,x}(\dt)\,\big|\nu\big(B_\dt(x)\big)\big|
\lesssim \frac{t}{\dt^{d+1}}\,(|x-x_0|+\dt)^d\cdot\e\,\lesssim \,\e,
\]
since $|x-x_0|\leq \al t < \dt$ (and also $t< \dt$). To estimate the integral in \eqref{bdry} we split it as $\int_0^t + \int_t^\dt$.
In the first range we use that
\[
\big|\vrho_x(r)\big|\leq\,(|x-x_0|+r)^d\,\e\lesssim \e\,t^d,
\]
since $|x-x_0|\leq \al t$ and $r\leq t$. So, 
\Beas 
\int_0^t|\vrho_x(r)|\,\Big|\frac{d\cK^0_{t,x}}{dr}(r)\Big|\,dr & \lesssim & \e\, t^d\, \int_0^t\Big|\frac{d\cK^0_{t,x}}{dr}(r)\Big|\,dr\\
& = & \e\, t^d\, \Big[\cK^0_{t,x}(0)-\cK^0_{t,x}(t)\Big]\\
& \leq & \e\, t^d\,\cK^0_{t,x}(0)\,\lesssim\,\e,
\Eeas
with the last bound due to \eqref{Kro}. Finally, if $r\in(t,\dt)$ we have
\[
\big|\vrho_x(r)\big|\leq\,(|x-x_0|+r)^d\,\e\leq (\al t+r)^d\,\e\,\lesssim \e\,r^d.
\]
Hence
\Beas 
\int_t^\dt|\vrho_x(r)|\,\Big|\frac{d\cK^0_{t,x}}{dr}(r)\Big|\,dr & \lesssim & \e\, \int_0^\dt\Big|\frac{d\cK^0_{t,x}}{dr}(r)\Big|\,r^d\,dr\\
\mbox{{\tiny (parts)}} & = & \e\, \Big(-\Big[r^d\,\cK^0_{t,x}(r)\Big]_0^\dt+d\int_0^\dt \cK_{t,x}(r)\,r^{d-1}\,dr\Big)\\
\mbox{{\tiny (by \eqref{Kro})}} & \lesssim & \e\, \int_0^\infty\frac {t\,r^{d-1}\,dr}{(t+r)^{d+1}} \,=\,c\,\e.
\Eeas
This shows that, if $|x-x_0|\leq \al t$ and $t<t_2=\min\{\dt, \dt/\al\}$ then
\[
|A_2(t,x)|\leq C\,\e.
\]
Combining this with the previous estimates \eqref{e0}, \eqref{P1s} and \eqref{limA1}, we conclude that, given $\e>0$ there exists 
$\tau_0=\tau_0(\e)>0$ such that, if $|x-x_0|<\al t$ and $t\in(0,\tau_0)$ then 
\[
\big|\SP_t\nu(x)\big|<\e.
\]
Thus, \eqref{limP} holds and we have completed the proof of Theorem \ref{th2}.
\ProofEnd

\BR
\label{R_x0}
As observed above, in the special case that $x_0=0$, the estimate of the remainder term $\cR_{t,x}$ can be slightly improved, due to
the factor $|x|=\cO(t)$ appearing in \eqref{A1n}. Indeed, this will give a better factor $t^2$ in the estimate \eqref{auxA1},
which in turn implies that
\[
A_1(t,x)
 =\cO(1), \quad \mbox{as $t\searrow0$}
\]
for all dimensions $d\leq 4$. 
Moreover, in order to have, for dimensions $d\geq5$, the estimate 
\[
\limsup_{|x-x_0|<\al t\to0}|A_1(t,x)|\leq C\, \e,
\]
one only needs the assumption 
\Be
\lim_{r\to0}\frac{|\nu-\ell\, dy|(B_r(0))}{r^{d-4}}=0, 
\label{d4}
\Ee
which at the point $x_0=0$ is weaker than \eqref{d3}.
So, overall, we can formulate this special case as a separate theorem.

\begin{theorem}\label{th3}
Let $\nu\in\cM(\Phi)$ and assume that $0$ is a $\sigma$-point of $\nu$ with value $\ell$. 

\sline (i) If $d\in\{1,2,3, 4\}$ then
\Be
\lim_{|x|<\al t\to 0} \SP_t\nu(x) = \ell 
, \quad \forall\,\al>0.
\label{nt20}
\Ee
\sline (ii) If $d\geq5$ then \eqref{nt20} holds under the additional assumption \eqref{d4}.

\end{theorem}

\ER

\section{Declarations}
\begin{enumerate}
\item Conflict of interest: not applicable.
\item Availability of data and materials: not applicable.
\end{enumerate}

\end{document}